\documentclass[english]{amsart}

\usepackage{amsmath,amssymb,enumerate,bbm}

\usepackage[T1]{fontenc}
\usepackage[all]{xy}

\usepackage{babel}
\usepackage{amstext}
\usepackage{amsmath}
\usepackage{amsfonts}
\usepackage{latexsym}
\usepackage{ifthen}

\usepackage{xypic}
\xyoption{all}
\newcommand{\cal}{\mathcal}

\newcommand{\lra}{\longrightarrow}

\newtheorem{lemma1}{}[section]

\newenvironment{lemma}{\begin{lemma1}{\bf Lemma.}}{\end{lemma1}}

\newenvironment{theorem}{\begin{lemma1}{\bf Theorem.}}{\end{lemma1}}
\newenvironment{proposition}{\begin{lemma1}{\bf Proposition.}}{\end{lemma1}}
\newenvironment{corollary}{\begin{lemma1}{\bf Corollary.}}{\end{lemma1}}
\newenvironment{remark}{\begin{lemma1}{\bf Remark.}\rm}{\end{lemma1}}
\newenvironment{definition}{\begin{lemma1}{\bf Definition.}}{\end{lemma1}}

\newenvironment{conjecture}{\begin {lemma1}{\bf Conjecture.}}{\end{lemma1}}

\newenvironment{hypothesisH}{{\bf Hypothesis (H).}}{}

\newenvironment{remark*}{{\bf Remark.}}{}
\newenvironment{example*}{{\bf Example.}}{}

\newcommand{\Q}{\ensuremath{\mathbb{Q}}}
\newcommand{\Z}{\ensuremath{\mathbb{Z}}}
\newcommand{\C}{\ensuremath{\mathbb{C}}}
\newcommand{\N}{\ensuremath{\mathbb{N}}}
\newcommand{\PP}{\ensuremath{\mathbb{P}}}

\newcommand{\dimension}[1]{\ensuremath{\mbox{dim}(#1)}}

\newcommand{\merom}[3]{\ensuremath{#1:#2 \dashrightarrow #3}}

\newcommand{\holom}[3]{\ensuremath{#1:#2  \rightarrow #3}}
\newcommand{\fibre}[2]{\ensuremath{#1^{-1} (#2)}}

\makeatletter
\ifnum\@ptsize=0  \fi \ifnum\@ptsize=2  \fi 

\newcommand\sO{{\mathcal O}}

\newcommand{\chow}[1]{\ensuremath{\mathcal{C}(#1)}}

\newcommand{\upX}{\ensuremath{\tilde{X}}}
\newcommand{\upY}{\ensuremath{\tilde{Y}}}

\newcommand{\upE}{\ensuremath{\tilde{E}}}
\newcommand{\upB}{\ensuremath{\tilde{B}}}

\newcommand{\barX}{\ensuremath{\overline{X}}}
\newcommand{\barY}{\ensuremath{\overline{Y}}}

\newcommand{\barE}{\ensuremath{\overline{E}}}
\newcommand{\barB}{\ensuremath{\overline{B}}}

\newcommand{\barT}{\ensuremath{\overline{T}}}
\newcommand{\barDelta}{\ensuremath{\overline{\Delta}}}
\newcommand{\Hplus}{($\mbox{H}^+$) }

\title{Threefolds with quasi-projective universal cover}
\date{\today}


\author{Beno\^it Claudon}
\author{Andreas H\"oring}

\address{Beno\^it Claudon, Institut Fourier - UMR 5582 - 100, rue des Maths B.P. 74, 38402 Saint-Martin d'H\`eres, France}
\curraddr{Institut \'Elie Cartan Nancy, Universit\'e Henri Poincar\'e Nancy 1, B.P. 70239, 54506 Vandoeuvre-l\`es-Nancy Cedex, France}
\email{Benoit.Claudon@ujf-grenoble.fr}

\address{Andreas H\"oring, Universit\'e Pierre et Marie Curie and Albert-Ludwig Universit\"at Freiburg} 
\curraddr{Mathematisches Institut, Albert-Ludwigs-Universit\"at
  Freiburg, Eckerstra{\ss}e 1, 79104 Freiburg im Breisgau, Germany}
\email{hoering@math.jussieu.fr}

\begin{document}

\begin{abstract} 
We study compact K\"ahler threefolds $X$ with 
infinite fundamental group whose universal cover
can be compactified. Combining techniques from $L^2$-theory, Campana's geometric orbifolds 
and the minimal model program
we show that this condition imposes strong restrictions on the geometry of $X$.
In particular we prove that if a projective threefold with infinite fundamental group
has a quasi-projective universal cover $\upX$, then $\upX$ is isomorphic to the product
of an affine space with a simply connected manifold.
\end{abstract}

\maketitle


\vspace{-1ex}

\section{Introduction}

\subsection{Main results}

In his book \cite{Sha74}, Shafarevich splits the classification of universal covers of projective manifolds into two types. The type I corresponds to manifolds having a finite fundamental group and their universal covers still belong to the class of projective manifolds. The type II is somehow more delicate to define but in this case the fundamental group is infinite and the universal cover is very far from being projective.
The following question is thus quite natural : can the universal cover be an affine or more generally a quasi-projective variety ?
Tori give a well-known example where the universal cover is affine, but a very interesting
result of Nakayama shows that this is the only possible example:

\begin{theorem} \cite[Thm.1.4]{Nak99} \label{theoremnakayama}
Let $X$ be a projective manifold such that the universal cover $\upX$ contains no positive-dimensional 
compact subvarieties and admits an embedding $\upX \subset \barX$ as a Zariski open subset into a compact complex manifold $\barX$.
If the abundance conjecture holds, then $X$ is (up to finite \'etale cover) an abelian variety.
\end{theorem}

In this paper we would like to understand how the geometry of a compact K\"ahler manifold 
is governed under a more general assumption on the universal cover that includes quasi-projective manifolds.
More precisely we want to impose the following condition:

\begin{hypothesisH}{\em
\ $X$ is a compact K\"ahler manifold with infinite $\pi_1(X)$ and 
the universal cover $\upX$ admits an embedding 
$\upX \subset \barX$ as a Zariski open set of a compact complex manifold $\barX$.

We say that $X$ satisfies the {\bf Hypothesis \Hplus} if $\barX$ is a K\"ahler manifold. }
\end{hypothesisH}

Note that the more restrictive hypothesis \Hplus is still much more general than supposing that 
$\upX$ is quasi-projective.
By the uniformisation theorem the case of curves is easily settled : 
the elliptic curves are the only ones having at the same time an infinite fundamental group 
and a universal cover which can be compactified. For surfaces we show the following (which is slightly
more general than the result obtained in \cite{Cla10} by the first-named author).

\begin{theorem} \label{theoremclaudon}
Let $X$ be a compact K\"ahler surface satisfying (H). 
Then (up to finite \'etale cover) one of the following holds:
\begin{enumerate}[(1)]
\item $X$ is a ruled surface over an elliptic curve and $\upX \simeq \C \times \PP^1$;
\item $X$ is a torus and the universal cover is $\C^2$.
\end{enumerate}
\end{theorem}

For compact K\"ahler threefolds the hypothesis (H) does not lead to a complete classification,
but it still gives a strong restriction on the fundamental group.

\begin{theorem}\label{theoremalmostabelian}
Let $X$ be a compact K\"ahler threefold satisfying (H).
If $X$ is not uniruled suppose moreover that $X$ admits a minimal model\footnote{For projective threefolds Mori has shown the existence
of minimal models. For non-algebraic
compact K\"ahler threefolds the existence of minimal models is still open, cf. Section \ref{subsectionmori}.}.
Then the fundamental group $\pi_1(X)$ is almost abelian and
(up to finite \'etale cover) the Albanese map is a surjective map with connected fibres and trivial orbifold divisor (cf. Definition \ref{definitionorbifold}).
\end{theorem}

Using the compactness of the cycle space implied by the K\"ahler condition in the hypothesis \Hplus, we obtain a complete picture
of the situation.

\begin{theorem} \label{theoremmain}
Let $X$ be a compact K\"ahler threefold satisfying \Hplus. Then (up to finite \'etale cover) the Albanese map is a locally trivial fibration whose fibre $F$ is simply connected and satisfies $\kappa(F)=\kappa(X)$. Moreover, the universal cover of $X$ splits as a product:
$$\upX \simeq F \times \C^{q(X)}.$$
\end{theorem}

\subsection{Outline of the paper}

The guiding framework of this paper is the theory of geometric orbifolds introduced in \cite{Cam04b}. 
In Section \ref{subsectionorbifolds} we recall parts of this theory and show that Theorem \ref{theoremalmostabelian} generalises to arbitrary dimension if a conjecture of Campana on the fundamental group of special manifolds holds.
Since this is still an open problem (even in dimension three), our starting point will be the following observation: 

\begin{theorem} \label{theoremchi}
Let $X$ be a non-uniruled compact K\"ahler threefold satisfying (H).
Suppose moreover that $X$ admits a minimal model.
Then its Euler characteristic vanishes:
$$\chi(X,\mathcal{O}_X)=0.$$
\end{theorem}

The proof of this fact is a mixture of two types of arguments: we show first that $X$ has a smooth minimal model 
and its Euler characteristic is not positive (in the projective case this is Miyaoka's theorem \cite{Miy87}); we then resort to $L^2$ theory to prove the reverse inequality.

\noindent Our next step is the description of the fundamental group of $X$.
If $X$ is a minimal model with Kodaira dimension two this is a notoriously difficult problem:
if $f: X \rightarrow Y$ is the Iitaka fibration, 
the image of the fundamental group of a general fibre $F$ in $\mathrm{Ker}(\pi_1(X)\to\pi_1(Y))$ could have infinite index. 
In particular we do not learn much about $\pi_1(X)$ by looking at the topology of $F$ and $Y$.
In the situation of Nakayama's theorem the absence of compact subvarieties in the universal cover implies that
image of $\pi_1(F)$ is infinite, 
so Koll\'ar's landmark paper on Shafarevich maps \cite[Ch.6]{Kol93} shows that (up to finite \'etale cover) $X$ is birational to a group scheme.
 
It is at this point that Theorem \ref{theoremchi} is used in a crucial way:
we prove that the property $\chi(X,\mathcal{O}_X)=0$ 
implies that the Iitaka fibration is of general type (cf. Definition \ref{definitiongeneraltype}).
Then we can use Campana's generalisation of the Kobayashi-Ochiai theorem to geometric orbifolds
to prove that the image of $\pi_1(F) \simeq \Z^{\oplus 2}$ has finite index in $\pi_1(X)$.

\medskip

The proof of Theorem \ref{theoremmain} has three independent steps. We first observe that the hypothesis \Hplus gives very
precise information on the birational geometry of $X$ which allows us to prove the existence of a minimal model
in the non-algebraic case. We then establish some general lemmata on fibre spaces satisfying \Hplus
which easily imply that the Albanese map is smooth.  
The most challenging and interesting part of the proof is to show that the Albanese map $\alpha: X \rightarrow E$ over
an elliptic curve is locally trivial: if $X$ is projective and the general fibre a minimal surface, this is due to Oguiso and Viehweg \cite{OV01}. 
If the fibres of $\alpha$ are rational surfaces the problem is more delicate since there are well-known examples of
$\PP^1 \times \PP^1$ degenerating into a Hirzebruch surfaces $\PP(\sO_{\PP^1} \oplus\sO_{\PP^1} (-2))$
(the ``moduli space'' of Hirzebruch surfaces is not Hausdorff).
It turns out that the condition \Hplus 
imposes additional restrictions on the deformation theory of rational curves on $X$ which exclude these jumping phenomena.

Let us conclude with a comment on the K\"ahler condition in our statements.
It is natural to impose that $X$ is K\"ahler: as is well-known from Hodge theory 
being a K\"ahler manifold means exactly that the complex and the differentiable (i.e. topological) structure of $X$ are related.
Indeed there exist plenty of non-K\"ahler compact manifolds covered by quasi-projective varieties:
\begin{enumerate}[(1)]
\item Hopf manifolds (covered by $\C^n\backslash\{0\}$) are never K\"ahler,
\item nilmanifolds are covered by $\C^n$ and complex tori are the only K\"ahler manifolds in this class,
\item generalizing the previous example, the manifold $X=G/\Gamma$ (where $\Gamma$ is a cocompact lattice of the connected and simply connected linear algebraic group $G$) is never K\"ahler as soon as $G$ is not abelian.
\end{enumerate}

The K\"ahler assumption on $\barX$ seems to be less essential and we don't know any example that satisfies (H) but 
not \Hplus. Nevertheless the proof of Theorem \ref{theoremmain} relies heavily on 
the compactness of the cycle space $\chow{\barX}$ assured by the K\"ahler condition  \cite{Ca80}, \cite{Va86}.

{\bf Acknowledgements.} 
The main part of this work was done when the first named author visited
the IHP for the Thematic Term Complex Algebraic Geometry in the spring of 2010.
We  would like to thank the institute for support and for the excellent working conditions.
We would like to thank T. Peternell for discussions
on K\"ahler threefolds and pointing out the references \cite{Pet86, HL83}.

\section{Notation and basic results}

For standard definitions in complex algebraic geometry 
we refer to  \cite{Har77} or \cite{Kau83}.
We will also use the standard terminology of the minimal model program, cf. \cite{Deb01}, \cite{CP97}.
Manifolds and varieties are always supposed
to be irreducible.
Since it will appear quite frequently, let us recall the following terminology: a group G is said to be \emph{almost abelian} if it contains an
abelian subgroup of finite index.

A fibration is a proper surjective map \holom{\varphi}{X}{Y} with connected fibres
from a complex manifold onto a normal complex variety $Y$. 
The $\varphi$-smooth locus is the largest Zariski open subset $Y^* \subset Y$
such that for every $y \in Y^*$, the fibre $\fibre{\varphi}{y}$ is a smooth variety of dimension $\dim X - \dim Y$.
The $\varphi$-singular locus is its complement.
A fibre is always a fibre in the scheme-theoretic sense, a set-theoretic fibre is the reduction of the fibre.

Let $X$ be a complex manifold, and let $D \subset X$ be a  subvariety of $X$. Let
\holom{\nu}{\hat{D}}{D} be the normalisation of $D$, then we set 
$$\pi_1(D)_X:=\mathrm{Im}\left(\nu_*: \pi_1(\hat{D})\lra\pi_1(X)\right).$$

If \holom{\varphi}{X}{Y} is a fibration, 
we have a sequence of maps between fundamental groups:
$$1 \rightarrow \pi_1(F)_X \lra\pi_1(X)\lra\pi_1(Y)\lra1$$
where $F$ denotes a general $\varphi$-fibre. 
This sequence is exact in  $\pi_1(F)_X$  and $\pi_1(Y)$, but in general it is only right exact  $\pi_1(X)$.
If $Y$ is a curve standard arguments (see for instance \cite[App.C]{Cam98}) show that after finite \'etale cover 
the sequence is exact in $\pi_1(X)$.
An important result due to Nori \cite[Lemma 1.5]{Nor83} says that for a
fibration with trivial orbifold divisor (cf. Definition
\ref{definitionorbifold}) the sequence is exact in $\pi_1(X)$.

In a compact K\"ahler manifold $X$, the subvarieties $Z$ having finite $\pi_1(Z)_X$ are actually (contained in) the fibres of a canonical fibration attached to $X$.
\begin{definition}\label{defigammareduction}\cite{Cam94,Kol93}
Let $X$ be a compact K\"ahler manifold; there exists a unique almost holomorphic fibration
$$\merom{\gamma_X}{X}{\Gamma(X)}$$
with the following property: if $Z$ is a subvariety through a very general point $x\in X$ with finite $\pi_1(Z)_X$, then $Z$ is contained in the fibre through $x$.\\
This fibration is called the $\gamma$-reduction of $X$ (Shafarevich map in the terminology of \cite{Kol93}) and the $\gamma$-dimension of $X$ is defined by
$$\gamma\mathrm{dim}(X):=\mathrm{dim}\left(\Gamma(X)\right).$$
The manifold $X$ is said to have \emph{generically large fundamental group} if its $\gamma$-dimension is maximal, $i.e.$ when $\gamma\mathrm{dim}(X)=\mathrm{dim}(X)$.
\end{definition}

\subsection{Geometric orbifolds and the Kobayashi-Ochiai theorem} \label{subsectionorbifolds}

Let us recall some basic definitions on geometric orbifolds introduced in \cite{Cam04b}.
They are pairs $(X,\Delta)$ where $X$ is a complex manifold and $\Delta$ a Weil $\Q$-divisor; they appear naturally as bases of fibrations to describe their multiple fibres.

Let \holom{\varphi}{X}{Y} be a fibration between compact K\"ahler manifolds and consider $\vert\Delta\vert\subset Y$ the union of the codimension one components of the $\varphi$-singular locus. If $D\subset\vert\Delta\vert$ is one of these, consider the divisor $\varphi^*(D)$ on $X$ and let us write it in the following form:
$$\varphi^*(D)=\sum_j m_jD_j+R,$$
where $D_j$ is mapped onto $D$ and $\varphi(R)$ has codimension at least 2 in $Y$.
\begin{definition} \label{definitionorbifold}
The integer
$$m(\varphi,D)=\mathrm{inf}_j(m_j)$$
is called the multiplicity of $\varphi$ above $D$ and we can consider the $\Q$-divisor
$$\Delta=\sum_{D\subset \vert\Delta\vert}(1-\frac{1}{m(\varphi,D)})D.$$
The pair $(Y,\Delta)$ is called the \emph{orbifold base} of $\varphi$.
\end{definition}
The canonical bundle of such a pair $(Y,\Delta)$ is defined as $K_Y+\Delta$ and, being a $\Q$-divisor, it has a Kodaira dimension denoted by $\kappa(Y,\Delta)$; as in the absolute case, an orbifold is said to be of general type if $\kappa(Y,\Delta)=\dimension{Y}>0$.
\begin{definition} \label{definitiongeneraltype}
A fibration \holom{\varphi}{X}{Y} is said to be of general type if $(Y',\Delta')$ is of general type for some fibration \holom{\varphi'}{X'}{Y'} birationally equivalent\footnote{Actually the Kodaira dimension of the orbifold base $(Y,\Delta)$ is \emph{not} a birational invariant; see \cite[p.512-513]{Cam04b}} to $\varphi$.
\end{definition}
For our purpose, the notion of fibration of general type will be relevant through the use of the generalised Kobayashi-Ochiai Theorem \ref{KO} below; from the viewpoint developed in \cite{Cam04b}, it is a central notion which determines a class of manifolds.

\begin{definition} \label{definitionspecial}
A compact K\"ahler manifold $X$ is said to be \emph{special} if it has no general type fibration.
\end{definition}

Examples of special manifolds are given by rationally connected manifolds and manifolds with zero Kodaira dimension. In fact
one expects that a special manifold can be obtained as a tower of fibrations whose fibres should be rationally connected or have $\kappa=0$ (in the orbifold sense). This gives a strong support for the following conjecture.
\begin{conjecture} (Abelianity conjecture)\label{abelconj}
The fundamental group of a special manifold is almost abelian.
\end{conjecture}

Special and general type manifolds appear as the building blocks of K\"ahler geometry: each compact K\"ahler manifold can be divided into a special and a general type part.

\begin{theorem}\cite[section 3]{Cam04b}\label{thmcore}
Let $X$ be a compact K\"ahler manifold; there exists a unique almost holomorphic fibration
$$\merom{c_X}{X}{C(X)}$$
which is at the same time special (i.e. the general fibres are special manifolds) and of general type. This fibration is called the \emph{core} of $X$.
\end{theorem}

We come now to the statement of the Kobayashi-Ochiai theorem (in the orbifold setting) and to its consequences on the geometry of manifolds satisfying (H).
The following easy observation can be seen as the curve case of the Kobayashi-Ochiai theorem but for the convenience of the
reader we state and prove it separately.
\begin{lemma}\label{maptocurve}
Let $X$ be a compact K\"ahler manifold satisfying (H). Let $\holom{\varphi}{X}{C}$ be a surjective map onto
a smooth curve. Then $C$ has genus at most one.
\end{lemma}
\begin{proof}
If the genus of $C$ is at least 2, its universal cover is the unit disk and $\holom{\varphi}{X}{C}$ induces
a bounded holomorphic function on $\upX$. This function extends to $\barX$  and must then be constant (compactness of $\barX$), contradicting the surjectivity of $\varphi$.
\end{proof}

In the general case, the Kobayashi-Ochiai theorem can be seen as a description of the singularities of meromorphic map to general type varieties: such maps can not have essential singularities.
\begin{theorem} \cite[Thm.2]{KO75}\cite[Thm.8.2]{Cam04b}\label{KO}
Let $X$ be a compact K\"ahler manifold, $V$ be a complex manifold, and let $B \subset V$ be a proper closed analytic subset. Let \merom{h}{V \setminus B}{X} be a nondegenerate meromorphic map, \emph{i.e.} such that the tangent map
$T_{V \setminus B} \rightarrow T_X$  is surjective at at least one point $v \in V \setminus B$. Let us finally consider \merom{g}{X}{Y} a general type fibration defined on $X$. Then $f=g\circ h$ extends to a meromorphic map $V \dashrightarrow Y$.
\end{theorem}

In our situation the last statement can be used to describe the topology of general type fibrations defined on $X$.
\begin{corollary}\label{corollary fundamental groups}
Let $X$ be a compact K\"ahler manifold satisfying (H).
Let $\holom{\varphi}{X}{Y}$ be a general type fibration, and denote by $F$ general $\varphi$-fibre. Then the natural morphism 
$$\pi_1(F)\lra\pi_1(X)$$
has finite index image. In particular, the fundamental group of $Y$ is finite.
\end{corollary}
\begin{proof}
Consider the composite morphism $\upX\lra Y$ and its holomorphic extension $\overline{\varphi}:\barX\lra Y$ given by Theorem \ref{KO} (up to replacing $\barX$ by a suitable bimeromorphic model, we can always assume that the extension is holomorphic). Then, every connected component of $\pi_X^{-1}(F)$ is contained in a unique connected component of the corresponding fibre of $\overline{\varphi}$. By compactness there are finitely many such components and this number is exactly the index of the subgroup $\pi_1(F)_X$. Since we always have a surjection
$$\pi_1(X)/\pi_1(F)_X\twoheadrightarrow\pi_1(Y),$$
it gives the finiteness of $\pi_1(Y)$.
\end{proof}

An immediate consequence of this statement is that a manifold satisfying (H) can never be of general type.
Moreover it yields first general results on the structure of the manifolds satisfying (H).

\begin{corollary}\label{core}
Let $X$ be a compact K\"ahler manifold satisfying (H).
Let \merom{c_X}{X}{C(X)} be the core fibration (cf. Theorem \ref{thmcore}) and denote by $F$ the general fibre. The group $\pi_1(F)_X$ has finite index in $\pi_1(X)$ and $\pi_1(C(X))$ is finite. In particular, if 
$$
\dim C(X) \in \{\dim(X)-1,\dim(X)-2\},
$$ 
 its fundamental group is then almost abelian.
\end{corollary}

Special curves and surfaces have indeed almost abelian fundamental groups \cite[Thm.3.33]{Cam04}. 
We can remark here that the abelianity conjecture \ref{abelconj} applied to the fibres of $c_X$ implies that the fundamental group of {\em any} manifold satisfying (H) is almost abelian. In particular (up to finite \'etale cover) its Albanese map should be non trivial. The structure of this map can be partly described using the generalised Kobayashi-Ochiai theorem: 

\begin{corollary}\label{corollarysurjectivityAlbanese}
Let $X$ be a compact K\"ahler manifold satisfying (H). The Albanese map of $\holom{\alpha}{X}{\mathrm{Alb}(X)}$ 
is surjective with connected fibres and 
with trivial orbifold divisor. In particular the sequence
$$1 \rightarrow \pi_1(F)_X \lra\pi_1(X)\lra\pi_1(Y)\lra1$$
is exact.
\end{corollary}

\begin{proof}
If the Albanese map $\alpha_X$ is not surjective, its image $Y=\alpha(X)$ is a fibre bundle over a variety of general type $W$ which is also a subvariety of an abelian variety \cite[Thm.10.9, p.120]{Ue75}. Being embedded in a complex torus, the fundamental group of $W$ is infinite and this fact contradicts the previous corollary.

Consider now the Stein factorization
$$X\stackrel{\beta}{\lra}Y\stackrel{\gamma}{\lra}\mathrm{Alb}(X)$$
of the surjective map $\alpha$. Under the assumption $\kappa(Y)>0$, we can apply \cite[Thm. 23]{Kaw81} to obtain a fibration $Y\lra W$ onto a general type variety $W$, finite over a complex torus. We thus get a contradiction as above. The fact that $\kappa(Y)=0$ easily implies that $\alpha$ has connected fibres (by \cite[Thm. 22]{Kaw81} $\gamma$ is finite \'etale and the universal property of the Albanese map implies that $\gamma$ is an isomorphism).

To prove that the orbifold divisor $\Delta$ of $\alpha$ is trivial, we argue by contradiction and suppose that
$$\emptyset\neq\Delta\subset \mathrm{Alb}(X).$$
As an effective $\Q$-divisor on a complex torus, $\Delta$ is numerically equivalent to the pull-back of an ample $\Q$-divisor $D$ on a quotient torus $q:\mathrm{Alb}(X)\twoheadrightarrow A$ (\cite[Thm.5.1]{Deb99}). Since $q$ is submersive, the orbifold divisor of $q\circ\alpha$ has to contain $D$ and the composite map
$$q\circ\alpha:X\lra A$$
is then of general type (see \cite[proof of Prop.5.3]{Cam04b}). The contradiction is obtained as a final application of the Corollary \ref{corollary fundamental groups}.
\end{proof}

\subsection{Birational geometry of $X$.}

Since the difference between the hypothesis (H) and the situation of Nakayama's theorem  \ref{theoremnakayama}
is that we authorise the universal cover $\upX$ to contain positive-dimensional compact subvarieties, they will play
a central role in this paper. In this paragraph we study the restrictions imposed by (H) on the birational geometry. 
The following lemma is a generalisation of \cite[Lemma 2.1]{Cla10}.

\begin{lemma} \label{lemmadivisor}
Let $X$ be a compact K\"ahler manifold satisfying (H). Let $D \subset X$ be an integral divisor
such that $\pi_1(D)_X$ is finite. Then $D|_D$ is numerically trivial.
\end{lemma}

\begin{proof} We fix  $\barX$ a compactification of $\upX$.
We argue by contradiction and suppose that there exists a curve $C \subset D$ such that $D \cdot C \neq 0$.
Since $\pi_1(D)_X$ is finite we see that
$$
\fibre{\pi_{X}}{D}=\cup_{\alpha \in \pi_1(X)/\pi_1(D)_X} D_\alpha
$$ 
is an infinite disjoint union of finite \'etale covers of $D$. 
Note that each $D_\alpha$ is a {\em compact} analytic subvariety of $\upX$, so also of the compactification $\barX$.
Moreover we have
\[
N_{D_\alpha/\upX} \simeq \pi_X^* N_{D/X}
\]
for every $\alpha$. For every $\alpha$ there exists a curve $C_\alpha \subset D_\alpha$
such that $\pi_X(C_\alpha)=C$ so by the projection formula
\[
D_\alpha \cdot C_\alpha = D \cdot C \neq 0.
\]
Since the $D_\alpha$ are disjoint we have
$D_{\alpha'} \cdot C_\alpha=0$ for $\alpha' \neq \alpha$, so the divisors $D_\alpha$ 
are linearly independent in the Neron-Severi group of $\barX$. 
This contradicts the finite-dimensionality of $NS(\barX)$.  
\end{proof}

\begin{remark} \label{remarkstricttransform}
Let \holom{\mu}{X}{Y} be a birational morphism between compact manifolds,
and let $D \subset Y$ be an irreducible divisor such that there exists a covering family
of curves $(C_t)_{t \in T}$ such that $D \cdot C_t < 0$. Denote by $D' \subset X$ the strict transform
of $D$, then there exists a family
of curves $(C'_t)_{t \in T}$ such that $D' \cdot C'_t < 0$. In fact we have
$$
D' = \mu^* D - \sum_i a_i E_i,
$$
where $a_i \geq 0$ and the $E_i$ are exceptional divisors. If $C'_t$ is the strict transform of a general member of the
family  $(C_t)_{t \in T}$, then $C'_t \not\subset E_i$ for any $i$ and the statement follows from the projection formula.
\end{remark}

The elementary Lemma \ref{lemmadivisor} yields strong restrictions on the birational geometry of $X$:

\begin{corollary} \label{corollaryMMP2}
Let $X$ be a compact K\"ahler surface satisfying $(H)$.
Then $X$ is a relative minimal model, i.e. does not contain (-1)-curves. $\square$
\end{corollary}

\begin{corollary} \label{corollaryMMP3}
Let $X$ be a compact K\"ahler threefold satisfying $(H)$.
Let 
\[
X=X_1 \stackrel{\mu_1}{\dashrightarrow} X_2 \stackrel{\mu_2}{\dashrightarrow} X_3 \stackrel{\mu_3}{\dashrightarrow} 
\ldots \stackrel{\mu_r}{\dashrightarrow} X_{r+1}   
\]
be a sequence of elementary Mori contractions between compact K\"ahler varieties. 
Then for every $j \in \{ 1, \ldots, r \}$ the contraction
$\mu_j$ contracts a divisor $E_j \subset X_j$ onto a curve $B_j$. 
In particular all the varieties $X_j$ are smooth.
\end{corollary}

\begin{proof}
We proceed by induction on $j$ according to the following scheme:

a) $X_j$ is smooth, so the contraction $\mu_{j}$ is divisorial with exceptional divisor $E_j$.

b) $\mu_{j}$ contracts $E_j$ onto a curve, so $X_{j+1}$ is smooth.

Since $X_1=X$ is smooth the initial step is obvious. 
Suppose now that $X_j$ is smooth, then $\mu_{j}$ contracts a divisor $E_j$.
There exists a covering family of rational curves $(C_t)_{t \in T}$ such that $E_j \cdot C_t<0$.
Thus if $E_j' \subset X$ denotes the strict transform, then by Remark \ref{remarkstricttransform}
the restriction $E_j'|_{E'_j}$ is not numerically trivial. Thus by Lemma \ref{lemmadivisor},
the group $\pi_1(E_j')_X$ is infinite. Since we have
$$
\pi_1(E_j')_X \simeq \pi_1(E_j)_{X_j}
$$
the classification of divisorial Mori contractions
in dimension three (\cite{Mor82}, cf. Theorem \ref{theorempeternell} for the K\"ahler case) implies that $E_j$ is not 
contracted to a point. Thus $E_j$ is contracted onto a curve and $X_{j+1}$ is smooth.
\end{proof}

\section{Threefolds satisfying (H)}

\subsection{Euler characteristic of compact K\"{a}hler threefolds}

The following lemma shows first that the holomorphic Euler characteristic of compact K\"ahler threefolds behaves as in the projective case.

\begin{lemma}\label{lemmachithreefold}
If $X$ is a (non uniruled) compact K\"ahler threefold admitting a smooth minimal model, its Euler characteristic is non positive:
$$\chi(X,\mathcal{O}_X)\le0.$$
\end{lemma}
\begin{proof}
If $X$ is projective this is a well-known consequence of Miyaoka's semipositivity 
theorem for the cotangent bundle \cite{Miy87}. If $X$ has algebraic dimension zero, the statement is due to Demailly and Peternell \cite[Thm.6.1]{DP03}.
Thus we are left to deal with the case where $X$ is nonalgebraic of algebraic dimension at least one. In particular $X$ is not simple, so 
by \cite[Thm.1]{Pet01} the canonical bundle is semiample. 

If $\kappa(X)=0$ (thus $c_1(X)=0$ by abundance) the Beauville-Bogomolov decomposition implies the claim.

If $\kappa(X)=1$ the Iitaka fibration $\holom{\varphi}{X}{C}$ maps onto a curve and $K_X \simeq_\Q \varphi^* A$ for some ample $\Q$-divisor on $C$.
If $F$ is a general $\varphi$-fibre, we have
$$
F . c_2(X) = c_2(F) \geq 0
$$
since $F$ is a compact surface with non-negative Kodaira dimension. Thus by the Riemann-Roch formula for threefolds 
$$
\chi(X, \sO_X) = - \frac{1}{24} K_X \cdot c_2(X) = - \frac{1}{24} \varphi^* A \cdot c_2(X) \leq 0.
$$

If $\kappa(X)=2$ the Iitaka fibration $\holom{\varphi}{X}{S}$ maps onto a surface and $K_X \simeq_\Q \varphi^* A$ for some ample $\Q$-divisor on $S$.
For $m \gg 0$, let $D \in |mA|$ be an effective, smooth very ample divisor such that $X_D:=\fibre{\varphi}{D}$ is a smooth elliptic surface.
By the adjunction formula we have $K_{X_D} \simeq_\Q \varphi|_{X_D}^* (m+1) A|_D$
where \holom{\varphi|_{X_D}}{X_D}{D} is the restriction of $\varphi$ to $X_D$.
Furthermore we have $\varphi^* D \cdot c_2(X) = c_2(T_X|_{X_D})$ and by the exact sequence
$$
0 \rightarrow T_{X_D} \rightarrow T_X|_{X_D} \rightarrow N_{X_D/X} \rightarrow 0
$$
we get
$$
c_2(T_X|_{X_D}) = c_2(T_{X_D}) + -K_{X_D} \cdot N_{X_D/X}.
$$
Since $N_{X_D/X} \simeq \varphi|_{X_D}^* D|_D$  and $K_{X_D}$ are pull-backs from the curve $D$, 
the second term is zero. The first term is non-negative since
$X_D$ is a compact surface with non-negative Kodaira dimension. We conclude again by Riemann-Roch.
\end{proof}

\subsection{Tools from $L^2$ theory and proof of Theorem \ref{theoremchi}}

Let us point out here the relevant facts on $L^2$ theory we shall use in the argument below; for general references, see \cite[Ch.3]{MaMa}. Let $X$ be a compact K\"{a}hler manifold (with K\"{a}hler form $\omega$); consider the spaces of holomorphic $p$-form on $\upX$ which are $L^2$ with respect to $\tilde{\omega}$ (metric induced by $\omega$):
$$H^0_{(2)}(\upX, \Omega^p_{\upX}):=
\left\{f\in H^0(\upX, \Omega^p_{\upX})\vert \int_{\upX}f\wedge \overline{f}\wedge\tilde{\omega}^{n-p}<+\infty \right\}.$$
There are obvious inner products which turn them into Hilbert spaces (when $p=\dim X$ it is the Bergman space of $\upX$). 
If $(\sigma^{(p)}_j)_{j\ge1}$ denotes any Hilbert basis of $H^0_{(2)}(\upX, \Omega^p_{\upX})$, we can form the following kernels:
$$k^{(p)}(x)=\sum_{j\ge1}\vert\sigma^{(p)}_j(x)\vert^2, \quad x\in\upX,$$
the norms being computed with respect to $\tilde{\omega}$. It is a well-known fact that the sums converge and define smooth functions, invariant under the action of $\pi_1(X)$. With this in mind, we can define the following $L^2$ Hodge numbers:
$$h^0_{(2)}(\upX, \Omega^p_{\upX}):=\int_{\cal D}k^{(p)}(x)dV_{\tilde{\omega}},$$
where $\cal D$ is a fundamental domain of the universal cover. These invariants are related to those of $X$ in a rather subtle manner through the $L^2$ index theorem of Atiyah.
\begin{theorem}\cite{At76}\label{theoremL2index}
Let $X$ be a compact K\"{a}hler manifold of dimension $n$. Its $L^2$ Euler characteristic
$$\chi_{(2)}(\upX, \mathcal{O}_{\upX}):=\sum_{p=0}^n (-1)^p h^0_{(2)}(\upX, \Omega^p_{\upX})$$
equals the usual one:
$$\chi_{(2)}(\upX, \mathcal{O}_{\upX})=\chi(X,\mathcal{O}_X).$$
\end{theorem}

The $L^2$ Hodge numbers can be seen as a measure of the size of the $\pi_1(X)$-modules $H^0_{(2)}(\Omega^p_{\upX})$ but have (almost) nothing to do with the dimension of the underlying vector spaces as shown by the following proposition.
\begin{proposition}\label{propBergman}
Let $X$ be a compact K\"{a}hler manifold with infinite fundamental group. If $h^0_{(2)}(\upX, \Omega^p_{\upX})>0$, the space $H^0_{(2)}(\upX, \Omega^p_{\upX})$ is infinite dimensional in the usual sense. In particular, if $X$ is a compact K\"{a}hler manifold satisfying $(H)$, the Bergman space of $\upX$ is zero dimensional: $h^0_{(2)}(\upX, K_{\upX})=0$.
\end{proposition}
\begin{proof}
We argue by contradiction and assume that $H^0_{(2)}(\upX, \Omega^p_{\upX})$ is finite dimensional. Let us then consider $(\sigma^{(p)}_j)_{j=1..N}$ an orthonormal basis of the latter and compute:
$$N=\sum_{j=1}^N \int_{\upX}\vert\sigma^{(p)}_j\vert^2 dV_{\tilde{\omega}}=\int_{\upX}k^{(p)}dV_{\tilde{\omega}}
=\sum_{\gamma\in\pi_1(X)}\int_{\gamma(\cal D)}k^{(p)}dV_{\tilde{\omega}}
=\left\vert\pi_1(X)\right\vert h^0_{(2)}(\Omega^p_{\upX}).$$
Since the fundamental group of $X$ is assumed to be infinite, we get a contradiction.

To deduce the last conclusion, we have to use a special feature of the Bergman space: the integrability condition does not involve the metric when $p=\dim X$. If $\barX$ is a compactification of $\upX$ the canonical $L^2$ forms defined on $\upX$ extend through the boundary and this yields an identification
$$H^0_{(2)}(\upX, K_{\upX})=H^0(\barX,K_{\barX}).$$
This last space being finite dimensional the $L^2$ Hodge number $h^0_{(2)}(K_{\upX})$ has to vanish according to what precedes.
\end{proof}

We come now to our main observation in this paragraph: 
the holomorphic Euler characteristic is non negative for K\"ahler threefold with (H); this is achieved by using $L^2$ theory developed in the previous lines.

\begin{proof}[Proof of the Theorem \ref{theoremchi}]
By hypothesis $X$ admits a minimal model which by Corollary \ref{corollaryMMP3} is smooth.
Thus by
Lemma \ref{lemmachithreefold} above it is sufficient to show that $\chi(X,\mathcal{O}_X)\geq0$.

By Theorem \ref{theoremL2index}, the Euler characteristic is given by:
$$\chi(X,\mathcal{O}_X)=\chi_{(2)}(\upX, \mathcal{O}_{\upX})=h^0_{(2)}(\mathcal{O}_{\upX})-h^0_{(2)}(\Omega^1_{\upX})+h^0_{(2)}(\Omega^2_{\upX})-h^0_{(2)}(K_{\upX}).$$
The universal cover $\upX$ being non compact, there is no non zero $L^2$ holomorphic functions on $\upX$. The vanishing of the first $L^2$ Hodge number is given by Gromov's theorem below and Lemma \ref{maptocurve}. On the other hand, the vanishing of $h^0_{(2)}(K_{\upX})$ is a consequence of Proposition \ref{propBergman}. The $L^2$-index theorem is finally reduced to
$$\chi(X, \mathcal{O}_X)=\chi_{(2)}(\upX, \mathcal{O}_{\upX})=h^0_{(2)}(\Omega^2_{\upX})\ge0.$$
\end{proof}

For the sake of completeness, let us recall (a weak form of) Gromov theorem on $L^2$ forms of degree one.
\begin{theorem}\cite{Gro89}
Let $X$ be a compact K\"ahler manifold having non zero first $L^2$ Betti number. 
Then (up to finite \'etale cover) $X$ admits a fibration onto a curve of genus $g\ge2$.
\end{theorem}

For surfaces these arguments show that $L^2$ Hodge number has to be zero:

\begin{theorem} \label{theoremchisurface}
Let $X$ be a compact K\"ahler surface satisfying (H).
Then its Euler characteristic vanishes:
$$\chi(X,\mathcal{O}_X)=0.$$ 
\end{theorem}

We can now prove our first main statement.

\begin{proof}[Proof of Theorem \ref{theoremclaudon}]
By Corollary \ref{corollaryMMP2} the surface $X$ is a relative minimal model. Thus if $X$ is uniruled it is a ruled surface 
over a curve $C$ of genus at least one. By Lemma \ref{maptocurve} the curve has genus one and we are in case $(1)$ of the theorem.

Suppose now that $X$ is not uniruled, then $X$ is not of general type by the Kobayashi-Ochiai Theorem \ref{KO}.
Since $K_X$ is nef this implies $K_X^2=0$. Since $\chi(X, \sO_X)=0$ by Theorem \ref{theoremchisurface}
we know by Riemann-Roch that
$$
0 = \chi(X, \sO_X)= \frac{1}{12} (K_X^2+c_2(X))= \frac{1}{12} c_2(X).
$$ 
If $\kappa(X)=0$, we know from surface theory that $c_2(X)=0$ implies that $X$ is (up to finite \'etale cover) a torus
and we are in case $(2)$ of the theorem.

Thus we are left to exclude the case $\kappa(X)=1$. We know by \cite[III, Rem.11.5]{BHPV04}
that $c_2(X)=0$ implies that the Iitaka fibration is almost smooth, i.e. the reduction of every singular fibres 
is an elliptic curve. Up to taking a finite \'etale covering the Iitaka fibration $X \rightarrow C$ is smooth over a curve $C$ of genus at most one
(Lemma \ref{maptocurve}), hence locally trivial
by \cite[III,Thm.15.4]{BHPV04}. 
If $C$ is rational, $X$ is uniruled by \cite[V, Thm.5.4]{BHPV04}, a contradiction.
If $C$ is elliptic, $X$ is abelian or hyperelliptic by  \cite[Ch.V, 5.B)]{BHPV04}, a contradiction.
\end{proof}

\begin{remark}
In \cite{Cla10}, $L^2$ theory on $\upX$ was used to rule out the case of $X$ being of general type. 
This can be done more efficiently by Theorem \ref{KO}, but the property $\chi(X,  \sO_X)=0$ 
is still useful for the case $\kappa(X)=0, 1$.
\end{remark}

\subsection{Proof of Theorem \ref{theoremalmostabelian}}

By Corollary \ref{corollarysurjectivityAlbanese} we are left to show 
that the fundamental group is almost abelian.
By Corollary \ref{core} we can assume that $X$ is a special threefold in the sense of Definition \ref{definitionspecial}. 
By \cite[7.2(4)]{Cam04b} this implies that the fundamental group is almost abelian unless (maybe)
$X$ has Kodaira dimension two. Thus we suppose without loss of generality that $\kappa(X)=2$.

By hypothesis $X$ admits a minimal model which by Corollary \ref{corollaryMMP3} is smooth.
Moreover we have $\chi(X, \sO_X)=0$ by Theorem \ref{theoremchi}. Since this property 
as well as the fundamental group are invariant
under the MMP we will suppose without loss of generality that $X$ is a minimal model.

Denote now by $\holom{\varphi}{X}{Y}$
the Iitaka fibration, and let $A$ be an ample $\Q$-divisor on $Y$ such that
$K_X \sim_\Q \varphi^* A$. We claim that $\varphi$ is almost smooth in codimension 
one\footnote{We recall that a fibration is almost smooth in codimension one if there exists a codimension
two subset $Z \subset Y$ such that for all $y \in Y \setminus Z$, the reduction of the
fibre $\fibre{\varphi}{y}$ is smooth.}. 
Let $D \subset Y$ be a general
divisor in $|mA|$ such that $D$ and $X_D :=\fibre{\varphi}{D}$ are smooth.
By Theorem \ref{theoremchi} we know that $\chi(X, \sO_X)=0$.
Using the same computation as in the proof of Lemma \ref{lemmachithreefold} we see that this implies
$c_2(T_{X_D})=0$. Thus the topological Euler characteristic of the minimal elliptic surface $X_D$ is zero.
By \cite[III, Prop.11.4, Rem. 11.5]{BHPV04} this implies that the reduction of all the fibres of $X_D \rightarrow D$ are
smooth elliptic curves. Since $D$ is a sufficiently general ample divisor this implies the claim.
 
Since $\varphi$ is almost smooth, it identifies to the core fibration by Lemma \ref{lemmageneral} below.
Thus the fundamental group is almost abelian by Corollary \ref{core}. $\square$

\begin{lemma} \label{lemmageneral}
Let $X$ be a minimal compact K\"ahler threefold.
Suppose that $\kappa(X)=2$ and denote by $\holom{\varphi}{X}{Y}$
the Iitaka fibration. Suppose that $\varphi$ is almost smooth in codimension one.
Then $(Y, \Delta)$ is of general type and $\varphi$ identifies to the core fibration.
\end{lemma}

\begin{proof}
Let $Y_0 \subset Y$ be the maximal Zariski open subset such that $Y_0$ is smooth
and $X_0:=\fibre{\varphi}{Y_0} \rightarrow Y_0$ is almost smooth and
the $\varphi$-singular locus is a disjoint union of smooth curves $\cup_i D_i$.
Note that $Y \setminus Y_0$ is a finite set of points.
 
Denote by $m_i$ the multiplicity of the divisor $\varphi^* D_i$, then 
we set $F_i$ for the support of $\varphi^* D_i$ and by the canonical bundle formula \cite[V,Thm.12.1]{BHPV04}
\[
K_{X_0} \simeq \sO_{X_0} (\varphi^* K_{Y_0} + \varphi^* \varphi_* K_{X_0/Y_0} + \sum_i (m_i-1) F_i).
\]
Set now $m := \prod_j m_j$, then we have for all $l \in \N$ 
\[
\varphi_* \sO_{X_0} (l m K_{X_0}) 
\simeq  \sO_{Y_0} (l m K_{Y_0} + l m \varphi_* K_{X_0/Y_0}+ l  \sum_i (m_i-1) \prod_{j \neq i} m_j D_i).
\]
If $j: Y_0 \hookrightarrow Y$ denotes the inclusion,
the isomorphism above extends to   an injective map
\[
\varphi_* \sO_X(l m K_{X}) \rightarrow \sO_Y(l m K_{Y} + l m j_* \varphi_* K_{X_0/Y_0} + l  \sum_i (m_i-1) \prod_{j \neq i} m_j D_i).
\]
Since $K_X$ has Kodaira dimension two, this implies that the $\Q$-Weil divisor
\[
K_{Y} + j_* \varphi_* K_{X_0/Y_0} +  \sum_i \frac{m_i-1}{m_i} D_i
\] 
has Kodaira dimension two.
If $H \subset Y$ is general hyperplane section, then $H \subset Y_0$, so the induced fibration $X_H \rightarrow H$
is almost smooth. By \cite[III, Thm.18.2]{BHPV04} this implies that
$\varphi_* K_{X_0/Y_0}|_H$ is numerically trivial.  Since bigness is a numerical property we see that
$K_{Y} + \sum_i \frac{m_i-1}{m_i} D_i$ is big.

Thus $\varphi$ is a general type fibration onto a surface. Since $X$ is a threefold and not of general type,
$\varphi$ is the core fibration. 
\end{proof}

\section{Threefolds satisfying \Hplus}

\subsection{Geometry of fibrations on $X$.}

\begin{lemma} \label{lemmafibration}
Let $X$ be a compact K\"ahler manifold satisfying \Hplus. 
Suppose that $X$ admits an equidimensional fibration 
\holom{\varphi}{X}{Y} onto a compact K\"ahler manifold $Y$
such that the sequence 
$$
1 \rightarrow \pi_1(F)_X \rightarrow \pi_1(X) \rightarrow \pi_1(Y) \rightarrow 1
$$  
is exact and $\pi_1(F)_X$ is finite.
\begin{enumerate}[(1)]
\item Then $Y$ satisfies \Hplus.
\item Let $\Delta \subset Y$ be the $\varphi$-singular locus, and denote by \holom{\pi_Y}{\upY}{Y} the universal cover of $Y$. 
Let  $\barDelta$ be the Zariski closure of
$\fibre{\pi_Y}{\Delta}$ in some compactification $\upY \subset \barY$.
Then $\barDelta$ is a proper subset of $\barY$.
\end{enumerate}
\end{lemma}

\begin{proof} 
We fix a compactification $\upX \subset \barX$. Up to replacing $\barX$ by some bimeromorphic model we
can suppose that $\barX \setminus \upX$ is a divisor which we denote by $B$.

Since $\pi_1(F)_X$ is finite and  $\pi_1(X)$ is infinite, the group $\pi_1(Y)$ is infinite.
The exactness of the sequence of fundamental groups implies that the natural morphism \holom{\tilde{\varphi}}{\upX}{\upY}
induced by $\varphi$ between the universal covers is a fibration, the  general $\tilde{\varphi}$-fibre $F$ being an finite \'etale cover of a general $\varphi$-fibre.
The $\tilde{\varphi}$-fibres form a 
$\dim Y$-dimensional analytic family of compact cycles in $\upX \subset \barX$,
so by the universal property of the cycle space $\chow{\barX}$ there exists an injective map $\upY \hookrightarrow \chow{\barX}$.
The normal bundle of $F$ is trivial, so  
\[
h^0(F, N_{F/X}) = \dim Y.
\]  
The tangent space of $\chow{\barX}$ in the point $[F]$ identifies to $H^0(F, N_{F/X})$, 
thus the image of $\upY \hookrightarrow \chow{\barX}$
is contained in a unique irreducible component of $\chow{\barX}$ of dimension at most $\dim Y$. 
Since $\upY \hookrightarrow \chow{\barX}$ is injective, 
the component has dimension $\dim Y$ and we denote by  $\barY$ its normalisation. 
Thus we get an injective map $c: \upY \hookrightarrow \barY$ and since $\barY$ is normal it follows
by Zariski's theorem that it is an isomorphism onto its image. 

{\em Proof of the first statement.}
Let us now show that the image of $c$ is a Zariski open subset of $\barY$,
hence $\barY$ is a compactification of $\upY$.

Denote by $\holom{p}{\Gamma}{\barX}$ and $\holom{q}{\Gamma}{\barY}$ the universal family over
$\barY$.  Consider the following commutative diagram :
\[
\xymatrix{
X  \ar[d]_{\varphi}  & \upX \ar[d]_{\tilde{\varphi}} \ar[l]_{\pi_X} \ar @{^{(}->}[r]  & \barX & \Gamma \ar[l]_{p} \ar[d]_{q}
\\
Y & \upY \ar[l]_{\pi_Y} \ar @{^{(}->}[rr]^c  & & \barY
}
\]
The general $q$-fibre is smooth since the general $\tilde{\varphi}$-fibre is smooth. 
Thus the singular locus of $\Gamma$ does not surject onto $\barY$ 
and we know by generic smoothness that there exists a
Zariski open subset $Y_0 \subset \barY$ such that the fibres $\Gamma_y$ over $Y_0$ 
are smooth. Moreover there exist $q$-fibres that do not meet the boundary $B$ (e.g. the $\tilde{\varphi}$-fibres),
so up to replacing $Y_0$ by a Zariski open subset we can assume that it parametrizes compact K\"ahler manifolds
that are not contained in the boundary $B$. 
We claim that this implies $\Gamma_y \subset \upX$ for every $y \in Y_0$.

Indeed if $F$ is a general $\tilde{\varphi}$-fibre, then $F \subset \upX$ so $F$ is disjoint from $B$. 
Since the $\Gamma_y$ are deformations of $F$, we obtain
$c_1(B|_{\Gamma_y})=0$.
Since $\Gamma_y$ is not contained in $B$, the intersection $B \cap \Gamma_y$ defines a non-zero global section of
the line bundle $B|_{\Gamma_y}$. The manifold $\Gamma_y$ being K\"ahler the condition $c_1(B|_{\Gamma_y})=0$ 
implies that the section
does not vanish. This proves the claim.

Since $\Gamma_y \subset \upX$ for every $y \in Y_0$,  we can consider $\pi_X(\Gamma_y)$ which is a compact cycle 
of dimension $\dim F$ in $X$. 
Moreover $\pi_X(\Gamma_y)$ has the same cohomology class as $\pi_X(F)$ which is
a (multiple of a) general $\varphi$-fibre.  
Thus if $\omega_Y$ is a K\"ahler form on $Y$, we have
\[
[\pi_X(\Gamma_y)] \cdot \varphi^* [\omega_Y] = [\pi_X(F)] \cdot \varphi^* [\omega_Y] = 0,
\]
so $\pi_X(\Gamma_y)$ is contained in a fibre of the equidimensional map $\varphi$.
Since $[\pi_X(\Gamma_y)]=[\pi_X(F)]$ and $X$ is K\"ahler, the support of $\pi_X(\Gamma_y)$ is a $\varphi$-fibre.
Thus $\Gamma_y$ is a $\tilde{\varphi}$-fibre in particular it is parametrised by $\upY$. 
This shows the first statement.  

{\em Proof of the second statement.} The image of $\fibre{\pi_Y}{\Delta}$
under the embedding $c: \upY \hookrightarrow \barY$ is contained
in the $q$-singular locus, i.e. the analytic subset of $\barY$ parametrising singular cycles. Since
the general $q$-fibre is smooth, this locus is a Zariski closed proper subset of $\barY$.
\end{proof}

\begin{corollary} \label{corollaryfibrationone}
Let $X$ be a compact K\"ahler manifold satisfying \Hplus. 
Suppose that $X$ admits a fibration 
\holom{\varphi}{X}{Y} onto a curve $Y$ such that for a general fibre $F$ the group
$\pi_1(F)_X$  is finite. 
Then (up to finite \'etale cover) the fibration \holom{\varphi}{X}{Y} 
is a smooth map with simply connected fibres onto an elliptic curve.

Moreover, up to replacing $\barX$ by a suitable bimeromorphic model, the induced fibration 
$\holom{\tilde{\varphi}}{\upX}{\C}$ extends to a fibration $\holom{\overline{\varphi}}{\barX}{\PP^1}$
such that we have a commutative diagram
\[
\xymatrix{
X  \ar[d]_{\varphi}  & \upX \ar[d]_{\tilde{\varphi}} \ar[l]_{\pi_X} \ar @{^{(}->}[r]  & \barX  \ar[d]_{\overline{\varphi}}
\\
Y & \C \ar[l]_{\pi_Y} \ar @{^{(}->}[r]  & \PP^1
}
\]
\end{corollary}

\begin{proof}
After finite \'etale cover the sequence 
$$
1 \rightarrow \pi_1(F)_X \rightarrow \pi_1(X) \rightarrow \pi_1(Y) \rightarrow 1
$$  
is exact. Thus by Lemma \ref{lemmafibration}(1) the curve satisfies \Hplus, so it is an elliptic curve.
This implies that $\pi_1(X)$ is almost abelian of rank two \cite[Lemme A.0.1]{Cla07}, so up to finite \'etale cover, $\pi_1(X) \simeq \Z^{\oplus 2}$
and the map $\pi_1(X) \rightarrow \pi_1(Y)$ is an isomorphism.

The $\varphi$-singular locus $\Delta$ is empty, since otherwise by 
Lemma \ref{lemmafibration}(2) the infinite set
$\fibre{\pi_Y}{\Delta}$ is contained in 
a Zariski closed proper subset of some compactification $\upY \subset \barY$.

Recall that by Ehresmann's theorem the smooth proper map $\varphi$ is a topological fibre bundle.
Thus we have an exact sequence of homotopy groups
$$
\ldots \rightarrow \pi_2(Y) \rightarrow \pi_1(F) \rightarrow \pi_1(X) \rightarrow \pi_1(Y) \rightarrow 1,
$$
where $F$ is a fibre. It is well-known that the second homotopy group
of an elliptic curve is trivial, so we have an injection $\pi_1(F) \hookrightarrow \pi_1(X)$. 
Since $\pi_1(X) \rightarrow \pi_1(Y)$ is an isomorphism, the fibre $F$ is simply connected.

Let us now prove that the fibration 
$\holom{\tilde{\varphi}}{\upX}{\C}$ extends to a fibration $\holom{\overline{\varphi}}{\barX}{\PP^1}$.
We start with any compact K\"ahler compactification $\upX \subset \barX$.
Note first that by Picard's theorem the image of the injective map $\upY \simeq \C \hookrightarrow \PP^1 \simeq \barY \subset \chow{\barX}$
is surjective onto $\PP^1 \setminus \infty$. 
Denote by $\holom{p}{\Gamma}{\barX}$ and $\holom{q}{\Gamma}{\barY}$ the universal family over
$\barY$, then $p$ is a bimeromorphic morphism and we claim that the exceptional locus of $p$ does not meet $\upX$.
Equivalently we claim that the cycle $\Gamma_\infty$ does not meet any of the $\tilde{\varphi}$-fibres $F$ : indeed 
$\Gamma_\infty$ is a deformation of $F$, so
$$
[\Gamma_\infty] \cdot F = F^2 = 0.
$$
Since all the $\tilde{\varphi}$-fibres $F$ are irreducible, this implies that $F \cap \Gamma_\infty = \emptyset$
or $F \subset \Gamma_\infty$. Yet the second case is not possible, since the cohomology class of $\overline{\Gamma_\infty \setminus F}$
is zero.
\end{proof}

\begin{corollary} \label{corollaryfibrationtwo}
Let $X$ be a compact K\"ahler manifold satisfying \Hplus. 
Suppose that $X$ admits an equidimensional fibration 
\holom{\varphi}{X}{Y} onto a smooth compact K\"ahler surface $Y$ such that 
 the sequence 
$$
1 \rightarrow \pi_1(F)_X \rightarrow \pi_1(X) \rightarrow \pi_1(Y) \rightarrow 1
$$  
is exact and $\pi_1(F)_X$ is finite. 
Then (up to finite \'etale cover) the following holds:
\begin{enumerate}[(1)]
\item $Y$ is a torus or a ruled surface over an elliptic curve.
\item If $Y$ is a torus, the $\varphi$-singular locus is empty and the general fibre is simply connected.
\item If $Y$ is a ruled surface,  the $\varphi$-singular locus is a disjoint union of sections of the ruling.
\end{enumerate}
\end{corollary}

\begin{proof}
By Lemma \ref{lemmafibration}(1) the surface $Y$ satisfies \Hplus, so (1) is a consequence of Theorem \ref{theoremclaudon}. 

We fix a compactification $\upY \subset \barY$ and denote by $\Delta$ the $\varphi$-singular locus. 
By Lemma \ref{lemmafibration} the analytic variety $\fibre{\pi_Y}{\Delta}$ is contained in a proper Zariski closed subset $\barDelta$ 
of $\barY$. In particular $\barDelta$ has finitely many irreducible components. Denote by
$\Delta_1$ the one-dimensional irreducible components of $\Delta$, then $\Delta_1$ is smooth,
i.e. all its irreducible components are smooth and do not intersect. Indeed the cover $\fibre{\pi_Y}{\Delta} \rightarrow \Delta$
is infinite, but  the compact variety $\barDelta$ has at most finitely many singularities.
This also shows that $\Delta$ does not have connected components of dimension zero.  
Otherwise $\fibre{\pi_Y}{\Delta \setminus \Delta_1}$ would be a  countable non-empty set
contained in the compact analytic set $T:=\barDelta \setminus \fibre{\pi_Y}{\Delta_1}$.
Thus $T$ contains a curve, hence $\Delta$ has a one-dimensional component not contained in $\Delta_1$.

Let now  $C \subset \Delta$ be an irreducible component.
By what precedes  \fibre{\pi_Y}{C} has finitely many irreducible components, so the index of $\pi_1(C)_Y$ is finite. 
By Lemma \ref{maptocurve} the curve $C$ is not of general type, so $\pi_1(C)_Y$ is abelian of rank at most two.
In particular if $\Delta \neq \emptyset$, then $\pi_1(Y)$ is almost abelian of rank two. 

Thus $\Delta$ must be empty if $Y$ is a two-dimensional torus and as in the proof of Corollary \ref{corollaryfibrationtwo}
we see that the $\varphi$-fibres are simply connected.

Since the fundamental group of a ruled surface over an elliptic curve is abelian of rank two, we also obtain that $\Delta$
is a union of disjoint elliptic curves if $Y$ is ruled. These curves are finite \'etale covers of the base of the ruling, so up to
finite \'etale base change they are sections.
\end{proof}

For small irregularity we can now give a stronger version of Corollary \ref{corollarysurjectivityAlbanese}.

\begin{proposition}\label{prop-small-irregularity}
Let $X$ be a compact K\"ahler manifold of dimension $n$ 
satisfying \Hplus and with a free abelian fundamental group. If $q(X)\le2$, the Albanese map of $X$ is a smooth map with simply connected fibres.
\end{proposition}
\begin{proof}
By Corollary \ref{corollarysurjectivityAlbanese} the Albanese map  \holom{\alpha}{X}{A} is a fibration which induces an isomorphism between the fundamental groups. The case $q(X)=1$ is then covered by Corollary \ref{corollaryfibrationone}. In view of Corollary \ref{corollaryfibrationtwo}, to prove the case $q(X)=2$ it is sufficient to show that $\alpha$ is equidimensional. We argue by contradiction, then $\alpha$ maps a divisor $D$ onto a point $\alpha(D)$.
But this implies that\footnote{Let us briefly explain this well-known 
fact in the case where $A$ is projective. Let $H$ be an effective divisor passing through $\alpha(D)$, then
we can write $\alpha^* H=H'+D$ with $D \not\subset \mbox{supp} H'$ but $D \cap H' \neq 0$.
Since $\alpha^* H \cdot D = 0$ we have
\[
\omega^{n-2} \cdot (\alpha^* H)^2 = \omega^{n-2} \cdot \alpha^* H \cdot (H' + D) = \omega^{n-2} \cdot (H')^2 + \omega^{n-2} \cdot H' \cdot D,
\]
and developing the left hand side implies $\omega^{n-2} \cdot H' \cdot D = - \omega^{n-2} \cdot D^2$.
Since $H' \cap D$ is an effective non-zero cycle, we obtain the claim.}
\[
\omega^{n-2}|_D \cdot D|_D = \omega^{n-2} \cdot D^2 \neq 0
\]
for some K\"ahler form $\omega$ on $X$.
In particular $D|_D$ is not numerically trivial.
Since $\pi_1(X) \simeq \pi_1(A)$ the set 
$\fibre{\pi_X}{D}$ is a union of 
infinitely many disjoint divisors $D_j$ mapped isomorphically onto $D$. In particular 
$\pi_1(D)_X$ is trivial, a contradiction to Lemma \ref{lemmadivisor}.
\end{proof}

\subsection{MMP for compact K\"ahler threefolds} \label{subsectionmori}

While the MMP for projective threefolds has been completed by the work of Kawamata, Koll\'ar, Miyaoka, Mori
and others more than twenty years ago, we do not have a complete MMP for compact K\"ahler threefolds.
Nevertheless the work of Campana and Peternell \cite{CP97, Pet98, Pet01} provides important tools
which together with the assumption \Hplus imply the existence of minimal models.

Let $X$ be a compact K\"ahler manifold. An {\it elementary contraction} is defined to be a surjective map
with connected fibres $\holom{\varphi}{X}{X'}$ onto a normal complex variety
such that $-K_X$ is relatively ample and $b_2(X)=b_2(X')+1$.
In general $X'$ might not be a K\"ahler variety, in particular $\varphi$ is not necessarily
a contraction of an extremal ray in the cone ${\overline {NE}}(X).$ 
Recall that a compact K\"ahler manifold $X$ is simple if there is no proper compact subvariety through a
very general point of $X$. It is Kummer, if $X$ is bimeromorphic to a quotient $T /G$ where $T$ 
is a torus and $G$ a finite group acting on $T$ .

\begin{theorem} \cite[Thm.1, Thm.2]{Pet01}, \cite[Main Thm.]{Pet98} \label{theorempeternell}
Let $X$ be a  smooth non-algebraic compact K\"ahler threefold.
Assume that $X$ is not both simple and non-Kummer.

I. If $K_X$ is nef, it is semiample.

II. If $K_X$ not nef, $X$ has an elementary contraction $\holom{\varphi}{X}{X'}$.

The contraction is of one of the following types.
\begin{enumerate}[(1)]
\item $\varphi $ is a $\PP_1$- bundle or a conic bundle over
a smooth non-algebraic surface,
\item $\varphi$ is bimeromorphic contracting an irreducible divisor
$E$ to a point, and $E$ together with its normal bundle $N_{E/X}$ is one of the
following $$(\PP_2,\sO_{\PP^2}(-1)), (\PP_2,\sO_{\PP^2}(-2)), (\PP_1 \times \PP_1, \sO_{\PP^1 \times \PP^1}(-1,-1)), (Q_0,\sO_{Q_0}(-1)),$$
where $Q_0$ is the quadric cone, 
\item $X'$ is smooth and $\varphi$ is the blow-up of $X'$ along a
smooth curve.
\end{enumerate}
The variety $X'$ is (a possibly singular) K\"ahler space in all cases except possibly (3). 
\end{theorem}  

In our situation the simple, non-Kummer case is easily excluded: if $X$ is a simple threefold with infinite fundamental group,
its fundamental group is generically large. Thus $X$ is Kummer by \cite[Thm.1]{CZ05}. This shows the

\begin{corollary} \label{corollarystartMMP}
Let $X$ be a compact K\"ahler threefold satisfying \Hplus. Then $K_X$ is semiample 
or $X$ admits an elementary contraction $\holom{\varphi}{X}{X'}$.
\end{corollary}

\begin{proposition} \label{propositionthreecases}
Let $X$ be a compact K\"ahler threefold satisfying \Hplus.
Then (up to finite \'etale cover) exactly one of the following holds
\begin{enumerate}[(1)]
\item $K_X$ is semiample ; or
\item $X$ admits a fibre type contraction ; or  
\item $X$ is a blow-up $X \rightarrow X'$ along an elliptic curve. The fundamental group of $X$ is isomorphic to $\Z^{\oplus 2}$. 
The Albanese map is a smooth fibration with simply connected fibres onto an elliptic curve.
\end{enumerate}
\end{proposition}

\begin{proof}
By Corollary \ref{corollarystartMMP} we know that if we are not in the first two situations,
then $X$ admits a birational contraction \holom{\mu}{X}{X'}.  
By Corollary \ref{corollaryMMP3} the divisor $E$ is 
contracted onto a smooth curve $B$ and $X'$ is smooth.
The ruling of the divisor $E$ gives a family of rational curves $F$ such that $E \cdot F<0$.
Thus the restriction $E|_{E}$ is not numerically trivial and by Lemma \ref{lemmadivisor} the group $\pi_1(E)_X$ is infinite,
so $B$ has positive genus. 
Let $\upE \subset \upX$  be a lift of the submanifold $E \subset X$ to the universal 
cover\footnote{Note that $\upE$ is not necessarily isomorphic to the universal cover of $E$, but the image
of a natural map from the universal cover of $E$ to $\upX$.}. 
Then $\upE$ is a ruled surface over some \'etale cover $\upB$ of $B$ with general fibre $F$ (we can identify the fibres of 
the $\upE \rightarrow \upB$ and $E \rightarrow B$ since they are simply connected).

Fix now a compactification $\upX \subset \barX$.
The curve $\upB$ parametrises a family of compact subvarieties on $\barX$, so 
we have a canonical injection $\upB \hookrightarrow \chow{\barX}$.
Since 
$$
\sO_{\PP^1} \oplus \sO_{\PP^1}(-1) \simeq N_{F/X} \simeq N_{F/\upX} \simeq N_{F/\barX},
$$ 
we have
$h^0(F, N_{F/\barX}) = 1$. Thus the cycle space 
$\chow{\barX}$ is smooth of dimension one in the point $[F]$ and we denote by $\barB$
the unique irreducible component through $[F]$.  
As in the proof of Lemma \ref{lemmafibration} we see that
the image of $\upB$ in $\barB$ is Zariski open. 
Hence if $\barE \subset \barX$ denotes the divisor covered by the cycles parametrised by $\barB$,
we obtain a compactification $\upE \subset \barE$. 
Moreover the curve $B$ has genus one: otherwise the holomorphic map $\upB \rightarrow B$
would extend to a holomorphic map $\barB \rightarrow B$ by the Kobayashi-Ochiai theorem \ref{KO},
but $\upB \rightarrow B$ has infinite fibres.

We claim now that $\pi_1(E)_X$ has finite index in $\pi_1(X)$. 
Since $\pi_1(E) \simeq \pi_1(B) \simeq \Z^{\oplus 2}$
this implies that $\pi_1(X)$ is almost abelian of rank two.
Thus (up to finite \'etale cover) $\pi_1(X) \simeq \Z^{\oplus 2}$
and we conclude by Proposition \ref{prop-small-irregularity}.

{\em Proof of the claim:}
We argue by contradiction and assume that  $\pi_1(E)_X$ has infinite index in $\pi_1(X)$. 
This exactly means that the preimage $\pi_X^{-1}(E)$ has infinitely many connected components $(E_i)_{i\in I}$,
each connected component being isomorphic to the lift $\upE \subset \upX$ we discussed above. 
We have seen that every $E_i$ can be compactified as a divisor $E'_i$ in $\barX$.
Moreover if $F$ is a fibre of the ruling $E \rightarrow B$ and $F_i$ a lift of $F$ contained in $E_i$, then
$$
E'_i\cdot F_i=E\cdot F<0.
$$
Since for $i \neq j$ the divisors $E'_i$ and $E'_j$ can only meet in the boundary $\barX\setminus\upX$ and
$F_i \cap  (\barX\setminus\upX)= \emptyset$ we have
\[
E'_i\cdot F_j = 0 
\]
for $i \neq j$. Thus the divisors $(E'_i)_i$ are linearly independent 
in the Neron-Severi group of $\barX$, contradicting the finite dimensionality of $NS(\barX)$. 
\end{proof}

\begin{proposition} \label{propositionfullMMP}
Let $X$ be a compact K\"ahler threefold that is not uniruled satisfying \Hplus.
Then $X$ admits a smooth minimal model.

More precisely there exists a birational morphism \holom{\mu}{X}{X_{\min}} onto a smooth
compact K\"ahler threefold such that $K_{X_{\min}}$ is nef.
The morphism $\mu$ decomposes into a sequence of blow-ups along elliptic curves.
\end{proposition}

\begin{proof}
If $K_X$ is semiample the statement is trivial. 
Otherwise we know by Proposition \ref{propositionthreecases} 
that (up to \'etale cover) the Albanese is a smooth fibration \holom{\alpha}{X}{E} onto an elliptic curve $E$. 
Moreover $X$ is a blow-up \holom{\mu}{X}{X'} along a smooth elliptic curve $E'$ and the rigidity lemma yields a smooth fibration
\holom{\alpha'}{X'}{E} such that $\alpha= \alpha' \circ \mu$. 
We claim that $X'$ is still K\"ahler, the statement then follows by induction on the Picard number.

{\em Proof of the claim.}
If $X'$ is not K\"ahler it follows from \cite[Thm.14]{HL83}, \cite[Thm.2.5]{Pet86} 
that there exists a current $T$ of bidimension (1,1) such that $T \geq 0$ and
\[
[E'+T] = 0.
\]
Yet if $H$ is an ample divisor on the curve $E$, then $(\alpha')^* H \cdot E'>0$ since $E'$ surjects onto $E$.
Since  $(\alpha')^* H \cdot T \geq 0$ for every $T \geq 0$ we can't have $[E'+T] = 0$.
 \end{proof}

\subsection{Classification of threefolds satisfying \Hplus.}

\begin{theorem} \label{theoremclassification}
Let $X$ be a compact K\"ahler threefold satisfying \Hplus.
Then (up to finite \'etale cover) one of the following holds:
\begin{enumerate}[(1)]
\item $\pi_1(X) \simeq \Z^{\oplus 6}$ and $X$ is a torus.
\item $\pi_1(X) \simeq \Z^{\oplus 4}$ and $X$ is a $\PP^1$-bundle over a two-dimensional torus.
\item $\pi_1(X) \simeq \Z^{\oplus 2}$ and $X$ admits a smooth morphism onto an elliptic curve.
Furthermore there exists a birational morphism $\holom{\mu}{X}{X'}$ which is sequence
of blow-ups along elliptic curves mapping surjectively onto $E$ 
such that the Albanese morphism of $\alpha':X' \rightarrow E$ is a locally trivial fibration.
\end{enumerate}
\end{theorem}

{\bf Proof.}
By Theorem \ref{theoremalmostabelian} 
there exists a finite \'etale cover such that $\pi_1(X) \simeq \Z^{\oplus 2r}$
with $r \in \N$. 
Since the Albanese map \holom{\alpha}{X}{Alb(X)} is a fibration by Corollary \ref{corollarysurjectivityAlbanese}, we have $r \leq 3$.

If $r=3$ the Albanese map is birational. Since $X$ is minimal by Proposition \ref{propositionthreecases}
we see that $X$ is isomorphic to $Alb(X)$.

If $r=2$ we know by Proposition \ref{prop-small-irregularity} that $\alpha$ is a smooth map with simply connected fibres. 
Since a simply connected curve is $\PP^1$ we conclude. 

Thus we are left to deal with the case $r=1$. In this case we know by  
Proposition \ref{prop-small-irregularity} that
the Albanese map is a smooth map \holom{\alpha}{X}{E}
with simply connected fibres onto the elliptic curve $E$.

{\em I. $X$ is not uniruled.}

If $X$ is not a minimal model, we know by Proposition \ref{propositionfullMMP} that it admits
a smooth minimal model. 
More precisely we know 
that a contraction \holom{\mu}{X}{X'} of this MMP is the blow-up along a smooth elliptic 
curve $E'$ and there exists a smooth fibration \holom{\alpha'}{X'}{E} such that $\alpha= \alpha' \circ \mu$.
Hence the minimal model 
$X'$ is a compact K\"ahler threefold that admits a smooth fibration \holom{\alpha'}{X'}{E}
onto an elliptic curve.
In particular all the $\alpha'$-fibres are minimal surfaces. If $X'$ is projective 
we know by \cite[Thm.0.1]{OV01} that after finite \'etale base change $X'$ is a product $E \times F$.
For the compact K\"ahler case we will use the hypothesis \Hplus in a much stronger sense. 
We make another case distinction:

{\em (i) $\kappa(X)=2$.}
Let $F$ be a general fibre of the Iitaka fibration \holom{f}{X'}{Y}. By the proof of Theorem \ref{theoremalmostabelian}
we know that $\pi_1(F)_{X'}$ has finite index in $\pi_1(X')$, so the elliptic curve $F$ maps surjectively onto $E$.
Equivalently we can say
that if $D$ is a $\alpha'$-fibre, then $D$ surjects onto $Y$. Since $K_{X'}$ is the pull-back of an ample divisor on $Y$, we see that
$D$ is of general type. In particular $\alpha'$ is a projective morphism onto a curve, so $X'$ is projective.

{\em (ii) $\kappa(X)=1$.} We claim that the $\alpha'$-fibres have Kodaira dimension one.
If this is not the case, $\alpha'$ coincides with the Iitaka fibration and 
$K_{X'} \simeq_\Q (\alpha')^* A$ for some ample $\Q$-divisor on $E$. Since by Theorem \ref{theoremchi}
and Riemann-Roch one has $0=\chi(X', \sO_{X'})= - \frac{1}{24} K_{X'} \cdot c_2(X')$ we see that a general fibre $F$ satisfies
$c_2(F)=0$. Hence the general fibres are tori (up to finite \'etale cover), so not simply connected, a contradiction.

The same argument shows that the general fibres $T$ of the Iitaka fibration $\holom{f}{X'}{\PP^1}$ are complex tori and $X'$ admits an equidimensional fibration 
$$
\holom{\psi':= \alpha' \times f}{X'}{E \times \PP^1}.
$$
Note that if we restrict $\psi'$ to $T=\fibre{f}{y}$, we get an elliptic bundle $T \rightarrow E \times y$, so the moduli of the
elliptic curves in this family is constant.

Set now  \holom{\psi:= \psi' \circ \mu}{X}{E \times \PP^1}, then $\psi$ is a factorisation of the Albanese map \holom{\alpha}{X}{E}.
By what precedes the $\alpha$-fibres are elliptic surfaces with Kodaira dimension one.
We claim that all the $\alpha$-fibres are bimeromorphic to each other. This obviously implies that
the $\alpha'$-fibres are bimeromorphic to each other. Since they are minimal surfaces, they are isomorphic.
 
{\em Proof of the claim.} Let $\upX \subset \barX$ be some compactification, then
by Corollary \ref{corollaryfibrationone} we have a fibration \holom{\overline{\alpha}}{\barX}{\PP^1}
compactifying the pull-back of $\alpha$. 
Note also that up to blowing-up $\barX$ the pull-back of $\psi$ to the universal cover $\upX$ extends to an elliptic fibration
\holom{\overline{\psi}}{\barX}{\PP^1 \times \PP^1 \supset \C \times \PP^1} 
(corresponding to the relative Iitaka fibration of the general $\overline{\alpha}$-fibres). In particular $\fibre{(\mu \circ \pi_X)}{T}$ 
extends to a surface $\barT$ admitting
an elliptic fibration $\barT \rightarrow \PP^1$ that over $\C \subset \PP^1$ is smooth with constant moduli.
Thus if $\barT_{min} \rightarrow \PP^1$ is a relative minimal model, the canonical bundle formula shows 
that $\kappa(\barT_{min})=-\infty$. 
By the classification of surfaces there is only one type of compact K\"ahler surfaces that
are relative minimal models with an elliptic fibration and $\kappa=-\infty$:  products $\PP^1 \times C$
with $C$ some elliptic curve. 
Since the deformations of $T$ dominate $X'$ we obtain a covering family of rational curves on $\barX$
such that the general member meets an $\alpha$-fibre $F$ (or rather its lift to $\upX$) in a unique point.
Let $\Gamma$ be the universal family and denote by \holom{p}{\Gamma}{X} and \holom{q}{\Gamma}{S}
the natural maps. Since there is a unique rational curve through a general point of $\barT$, the map $p$ is bimeromorphic.
Thus for every $F$ we obtain a bimeromorphic map
$q|_{\fibre{p}{F}}: \fibre{p}{F} \rightarrow S$, hence any two $\alpha$-fibres are bimeromorphic.

{\em (iii) $\kappa(X)=0$ (hence $c_1(X')=0$).} 
The Beauville-Bogomolov decomposition theorem implies that (up to finite \'etale cover) $X'$ is a  
product $E\times S$ where $S$ is a K3 surface and $E$ an elliptic curve.

{\em II. $X$ is uniruled.}

Since the $\alpha$-fibres are uniruled and simply connected surfaces, they are rationally connected.
In particular $X$ is projective and we can run a relative MMP on $X$ over $E$.
As in Proposition \ref{propositionfullMMP} we see
that a contraction \holom{\mu}{X}{X'} of this MMP is the blow-up along a smooth elliptic 
curve $E'$ and there exists a smooth fibration \holom{\alpha'}{X'}{E} such that $\alpha= \alpha' \circ \mu$. 
Thus the MMP terminates with  $X'$ a smooth projective threefold that admits a smooth fibration \holom{\alpha'}{X'}{E}
and an elementary contraction of fibre type \holom{\psi}{X'}{Y}.
Moreover $\alpha'$ factors through $\psi$.

For the rest of the proof we denote by $\mathbb F_d$ the Hirzebruch surface $\PP(\sO_{\PP^1}\oplus \sO_{\PP^1}(-d))$.

{\em Case 1) $Y$ is a surface.} Since $\psi$ is a $\PP^1$-  or conic bundle we know by Corollary \ref{corollaryfibrationtwo}
that $Y$ is a torus or a $\PP^1$-bundle over $E$. Since $\Z^{\oplus 2} \simeq \pi_1(X') \simeq \pi_1(Y)$ we are in the second case.

{\em Case 1a)} If $\psi$ is a $\PP^1$-bundle, all the $\alpha$-fibres are Hirzebruch surfaces. 
It is straightforward to see that there exists at most finitely many points in $0 \in E$ such that 
the fibre $F_0$ is not isomorphic to the general fibre $F$. Moreover if $F$ is isomorphic
$\mathbb F_e$, then $F_0$ is isomorphic to $\mathbb F_d$
with $d>e$. We claim that if such an $0$ exists there exists an analytic neighbourhood of $0$ such that all 
the fibres in the neighbourhood are isomorphic to $\mathbb F_d$. Since the ``exceptional'' fibres
are at most finite, this shows that $\alpha$ is locally trivial.

{\em Proof of the claim.}
Denote by $C$ the exceptional section of $\mathbb F_d$, i.e. the unique curve with self-intersection $-d$.
Then by Lemma \ref{lemmahirzebruch} the deformations of $C$ dominate $E$. 
Since $C$ is integral, a small deformation is an irreducible curve $C_t$ having also self-intersection $-d$.
Since a Hirzebruch surface has a unique irreducible curve with negative self-intersection (its exceptional section), the
$\alpha$-fibre containing $C_t$ is isomorphic to $\mathbb F_d$.

{\em Case 1b)} If $\psi$ is a conic bundle, 
we know by Corollary \ref{corollaryfibrationtwo} that the $\psi$-singular locus $\Delta$ is a disjoint union of 
sections of the ruling $Y \rightarrow C$.
Since the contraction is elementary, $\Delta$ is irreducible. In particular every $\alpha$-fibre
is a conic bundle with a unique singular fibre, i.e. the blow-up of a Hirzebruch surface in one point.

Let $F$ be an $\alpha$-fibre isomorphic to the
blow-up of the surface $\mathbb F_d$ with $d \geq 2$ in a point that is not in the exceptional section. 
Denote by $C$ the strict transform of the exceptional section, i.e. the unique section with self-intersection $-d$.
By Lemma \ref{lemmahirzebruch} the deformations of $C$ dominate $E$. 
Thus a small deformation $F_t$ of $F$ also contains a rational curve $C_t$ having self-intersection $-d$ that is a section for the
conic bundle structure $F_t \rightarrow \PP^1$. Thus $F_t$ is the blow-up of $\mathbb F_d$
in some point that is not in the exceptional section.
Since the automorphism group of  $\mathbb F_d$ acts transitively on the complement of the exceptional section\footnote{The complement of the exceptional section and a fibre is isomorphic to $\C^2$ and the automorphisms of the plane which are induced from $\mathrm{Aut}(\mathbb F_d)$ consist in the following transformations:
$$(x,y)\mapsto (ax+P(y),by+c)$$
where $(a,b,c)\in\C$ and $P$ is a polynomial of degree less than $d$; these automorphisms act then transitively on this complement.}, all these surfaces are isomorphic.

{\em 2) $Y$ is a curve (i.e. $\alpha'=\psi$).} Then the $\alpha'$-fibres are del Pezzo surfaces. 
If no $\alpha'$-fibre contains a (-1)-curve, they are isomorphic to $\PP^2$ or $\PP^1 \times \PP^1$. 
Hence $\alpha'$ is locally trivial \cite{Mor82}.

Suppose now that there exists a fibre $F$ that contains a (-1)-curve. We claim that there exists a finite \'etale base change $E' \rightarrow E$ such that the relative Picard number of $X' \times_E E' \rightarrow E'$ is at least two. Assuming this for the time being, let us see how
to conclude. The relative Picard number of $X' \times_E E' \rightarrow E'$ is at least two and the anticanonical bundle
is relatively ample. Thus $X' \times_E E'$ admits an elementary Mori contraction that does not identify to $X' \times_E E' \rightarrow E'$.
If the contraction is of birational type our algorithm starts again. If the contraction is of fibre type it maps onto a surface, so we are
in case 1).  

{\em Proof of the claim.} The normal bundle of the (-1)-curve in $X'$ is $\sO_{\PP^1} \oplus \sO_{\PP^1}(-1)$,
so it deforms in a one-dimensional family parametrized by some curve $E'$. It is straightforward to see that
all the curves parametrised by $E'$ are (-1)-curves in some fibre, in particular there is a natural finite morphism $E' \rightarrow E$.
We claim that $E'$ is an elliptic curve, in particular the morphism $E' \rightarrow E$ is \'etale: the morphism $X \rightarrow X'$
is a sequence of blow-ups along elliptic curves, in particular the strict transform of the rational curves
parametrised by $E'$ is well-defined and their normal bundle is $\sO_{\PP^1} \oplus \sO_{\PP^1}(-d)$ with $d>0$.
We can now argue as in the proof of Proposition \ref{propositionthreecases} : the lift of the family to the universal cover $\upX$
compactifies to an analytic family in $\barX$. In particular the universal cover of the non-rational curve $E'$ is quasi-projective, so $E'$ is an elliptic curve. 

By the universal property of the fibre product the family of (-1)-curves parametrised by $E'$ lifts to $X \times_E E'$. Moreover
there exists a unique (-1)-curve parametrised by $E'$ in every fibre of $X' \times_E E' \rightarrow E'$.
Let $D \subset X' \times_E E'$ be the divisor covered by these (-1)-curves and $F$ a general fibre, then 
$D \cap F$ is an effective divisor. Thus if the relative Picard number is one, it is ample. Yet $D \cap F$ is a (-1)-curve, so $D|_F^2=-1$.
$\square$

\begin{lemma}\label{lemmahirzebruch}
Let $X$ be a compact K\"ahler threefold satisfying \Hplus. Suppose 
that $X$ admits a smooth fibration \holom{\varphi}{X}{E} over an elliptic curve
such that the general fibre is rationally connected.
Let $C \subset F$ be a smooth rational curve contained in an $\varphi$-fibre $F$.
Then the family of deformations of $C$ dominates $E$.
\end{lemma}

\begin{proof}
If $C^2 \geq -1$ this is standard deformation theory, so
we can suppose that $-d:=C^2 \leq -2$, hence $-K_F \cdot C \leq 0$ and $F$ is not a del Pezzo surface.
By Corollary \ref{corollaryfibrationone} we have a commutative diagram
\[
\xymatrix{
X  \ar[d]_{\varphi}  & \upX \ar[d]_{\tilde{\varphi}} \ar[l]_{\pi_X} \ar @{^{(}->}[r]  & \barX  \ar[d]_{\overline{\varphi}}
\\
E & \C \ar[l]_{\pi_Y} \ar @{^{(}->}[r]  & \PP^1
}
\]

The fibration $\overline{\varphi}$ has infinitely many fibres $X_\alpha$ isomorphic to $F$, in particular
we have infinitely many curves $C_{\alpha}$ corresponding to lifts of $C$ to the universal cover.
If one of the rational curves $C_\alpha$ deforms in $\upX$, the curve $C$ deforms in $X$.
We argue by contradiction that this is not the case: then each of the points $[C_\alpha] \in \mathcal H(\barX)$
is a connected component of the Hilbert scheme $\mathcal H(\barX)$.
Since we have infinitely many such curves we obtain a contradiction if we construct an ample
divisor $H$ on $\barX$ such that $H \cdot C_\alpha$ does not depend on $\alpha$ : indeed it is well-known
that for a fixed degree, the Hilbert scheme $\mathcal H(\barX)$ has only finitely many irreducible components.

In order to construct $H$ we argue by induction on the Picard number $\rho(F)$ of the fibre $F$.
Since $F$ is not Fano the induction starts with $\rho(F)=2$, i.e. $F$ is a Hirzebruch surface $\mathbb F_d$ 
and $C$ the unique section with self-intersection $-d$. We run a MMP on $\barX$ over $\PP^1$ and claim that the birational contractions of
such a MMP take place in the fibre over infinity.

{\em Proof of the claim.} Since all the $\tilde{\varphi}$-fibres
are smooth, it is clear that if such a MMP contracts a divisor to a point, then this divisor is contained
in the fibre over infinity. In particular $\upX$ remains in the smooth locus of the varieties, so if the exceptional
locus of a contraction is not contained in  $\barX \setminus \upX$, deformation theory on threefolds (cf. \cite[II,Thm.1.13]{Kol96}) 
shows that the contraction
is divisorial. Thus we are left to discuss contractions of a divisor $E$ onto a curve $B$: the restriction to every $\tilde{\varphi}$-fibre is then the 
blow-down of (-1)-curves. Since for $d \geq 2$ the Hirzebruch surface $F \simeq \mathbb F_d$ does not contain (-1)-curves
this is not possible.

Thus up to replacing $\barX$ by another compactification (with terminal, $\Q$-factorial singularities, but
$\upX$ contained in the smooth locus) we can assume that $\barX$ admits a fibre type contraction \holom{\psi}{\barX}{Y} over $\PP^1$. 
Since $F$ is not Fano, $Y$ is a surface that is generically a $\PP^1$-bundle over $\PP^1$.
Since $-K_{\barX}$  is $\psi$-ample we can choose an ample divisor $L$ on $Y$ such that
$H:=-K_{\barX}+\psi^* L$ is ample. The restriction of $L$ to a $\tilde{\varphi}$-fibre
depends only on the degree $e$ of $L$ on the corresponding fibre of  $Y \rightarrow \PP^1$.
Since this degree is constant in flat families, we see that the restriction of $H$ to any fibre $X_\alpha$
is isomorphic to $-K_{X_\alpha}+\psi^*  \sO_{\PP^1}(e)$. In particular $H \cdot C_\alpha$ is constant.

For the induction step we again run a MMP on $\barX$ over $\PP^1$. Arguing as before we see that if the exceptional locus of 
a birational contraction $\mu: \barX \rightarrow \barX'$ is not contained in $\barX \setminus \upX$, 
it contracts a divisor $E$ onto a curve $B$. Moreover the restriction 
to every $\tilde{\varphi}$-fibre is the contraction of (-1)-curves. In particular the fibration $\overline{X}' \rightarrow \PP^1$
is still smooth over $\C$ and the curves $C_\alpha$ are not contracted by $\mu$ (they are not (-1)-curves).
By the induction hypothesis there exists a polarisation $H'$ such that the curves $\mu(C_\alpha)$ have constant degree with respect to $H'$.
The divisor $-E$ is $\mu$-ample, so $-E+m \mu^* H'$ is ample for $m \gg 0$ and has constant degree on the curves $C_\alpha$. 
 \end{proof}

\subsection{Proof of Theorem \ref{theoremmain}}

The splitting mentioned in the statement of Theorem \ref{theoremmain} is a straightforward consequence of deep results of Grauert.
\begin{theorem} \label{theoremisotrivial}
Let \holom{f}{X}{Y} be a locally trivial 
proper fibration between complex manifolds. If the universal cover of $Y$ is Stein and contractible, then the universal cover of $X$ splits as a product:
$$\upX\simeq \tilde{F}\times\upY.$$
\end{theorem}
\begin{proof}
Since $f$ is locally trivial and proper, 
it is a fibre bundle with fibre $F$ and group $G=\mathrm{Aut}(F)$ (a complex Lie group). Consider the fibre product
$$\upX_f=X\times_Y \upY;$$
it is a connected cover of $X$ which is also a fibre bundle over $\upY$ (fibre $F$ and group $G$). We can now apply \cite[Satz 6]{Gra58}: this fibre bundle has to be trivial and this gives a splitting
$$\upX_f\simeq F\times\upY.$$
Since $\upX_f$ is an intermediate cover, $\upX$ has to split as well.
\end{proof}

\noindent We can now combine the previous statement with the local triviality of the Albanese map of a minimal model of $X$ to conclude the proof in the general case.

\begin{proof}[Proof of Theorem \ref{theoremmain}]
If $X$ is a torus or a $\PP^1$-bundle over a torus we are done by Theorem \ref{theoremisotrivial} above.
Thus by Theorem \ref{theoremclassification} we are left to deal with the case where $X$ admits a birational
morphism  $\holom{\mu}{X}{X'}$ which is a sequence of blow-ups along elliptic curves  and the Albanese fibration
\holom{\alpha'}{X'}{E} is locally trivial with simply connected fibre $S$.
In particular by Theorem \ref{theoremisotrivial} the universal cover of $X'$ is a product $\C \times S$.
We claim that the universal cover $\upX$ is a product $\C \times \hat{S}$ where
$\hat{S} \rightarrow S$  is a sequence of blow-ups of points. 
This shows the statement on the structure of the universal cover,
moreover it shows that all the fibres of the Albanese morphism $\holom{\alpha}{X}{E}$ 
are isomorphic to $\hat{S}$. Thus $\alpha$ is locally trivial by the theorem of Fischer and Grauert.

{\em Proof of the claim.} 
We prove the statement in the situation where $\holom{\mu}{X}{X'}$ is the blow-up of a unique elliptic curve $C$, the general statement follows by induction on the number of blow-ups. 

Since $C$ maps surjectively onto $E$ we can suppose (up to finite \'etale cover) that $C$ is a section  $\holom{s}{E}{X'}$ of $\alpha'$.
Since the universal cover of $X'$ is isomorphic to the fibre product $X' \times_E \C$ we obtain a unique induced section $\holom{\tilde{s}}{\C}{\C \times S}$. The surface $\C \times S$ compactifies to $\PP^1 \times S$, so up to choosing a suitable compactification $\upX \setminus \barX$ we obtain a commutative diagram
 \[
\xymatrix{
X  \ar[d]_{\mu}  & \upX \ar[d]_{\tilde{\mu}} \ar[l]_{\pi_X} \ar @{^{(}->}[r]  & \barX  \ar[d]^{\overline{\mu}}
\\
X'  \ar[d]_{\alpha'}  & \C \times S \ar[d] \ar[l]_{\pi_{X'}} \ar @{^{(}->}[r]  & \PP^1 \times S  \ar[d]
\\
E \ar @/_/ [u]_s & \C \ar @/_/ [u]_{\tilde{s}} \ar[l]_{\pi_E} \ar @{^{(}->}[r]  & \PP^1 \ar @/_/ @{.>} [u]_{\overline{s}}
}
\]
Since $\upX$ is the blow-up of $\C \times S$ along $\tilde{s}(\C)$, the Zariski closure of $\tilde{s}(\C)$
is contained in the image of the exceptional locus of $\holom{\overline{\mu}}{\barX}{\PP^1 \times S}$.
Since $\PP^1 \times S$ is a threefold the irreducible components of this locus have dimension at most one,
so $\tilde{s}$ compactifies to a section $\holom{\overline{s}}{\PP^1}{\PP^1 \times S}$.
Choose now $e \in E$ arbitrary,
then for all $c \in \fibre{\pi_E}{e}$ one has $\tilde{s}(c)=s(e)$. Hence the compact curves 
$\overline{s}(\PP^1)$ and $\PP^1 \times s(e)$ meet in infinitely many points, so they are identical.

In particular we see that $\tilde{s}(\C)=\C \times s(e)$, so $\upX \simeq  \mbox{Bl}_{s(\C)} (\C \times S) \simeq \C \times \mbox{Bl}_{s(e)} S$. 
\end{proof}

\end{document}